\documentclass[11pt]{amsart}

\title{Amenability and Uniqueness}

\author{A. Ciuperca}

\email{alin.ciuperca@gmail.com}

\author{T. Giordano}

\address{Department of Mathematics and Statistics\\
University of Ottawa\\
585 King Edward Avenue\\
Ottawa, Ontario\\
KiN 6N5\\
Canada}

\email{giordano@uottawa.ca}

\author{P. W. Ng}

\address{Department of Mathematics\\
University of Louisiana at Lafayette\\
217 Maxim D. Doucet Hall\\
P. O. Box 41010\\
Lafayette, Louisiana\\
70504-1010\\
USA}

\email{png@louisiana.edu}

\author{Z. Niu}

\address{Department of Mathematics and Statistics\\
Memorial University of Newfoundland\\
St. John's, NL\\
A1C 5S7\\
Canada}

\email{zniu@mun.ca}

\thanks{A.C. thanks AARMS for postdoctoral support. T.G. and Z.N. were
partially supported by a grant from NSERC Canada.}

\newtheorem{thm}{Theorem}[section]
\newtheorem{prop}[thm]{Proposition}
\newtheorem{lem}[thm]{Lemma}
\newtheorem{cor}[thm]{Corollary}
\newtheorem{df}{Definition}[section]
\theoremstyle{definition}
\newtheorem{rem}[thm]{Remark}

\newcommand{\A}{\mathcal{A}}
\newcommand{\B}{\mathcal{B}}
\newcommand{\C}{\mathcal{C}}
\newcommand{\F}{\mathcal{F}}
\newcommand{\K}{\mathcal{K}}
\newcommand{\G}{\mathcal{G}}
\newcommand{\Mul}{\mathcal{M}}
\newcommand{\M}{\mathbb{M}}
\newcommand{\N}{\mathcal{N}}

\newcommand{\h}{\mathcal{H}}
\newcommand{\s}{\mathcal{S}}

\def \bib(#1;#2;#3;#4;#5;#6) {{#1}, {\it #2} {#3},
{\bf#4} (#5) {#6}\par\smallskip}

\date{\today}

\numberwithin{equation}{section}

\begin{document}

\maketitle

\section{Introduction}

One of the most important achievements up to now in the theory of operator algebras is the complete classification up to isomorphism of 
injective von Neumann factors with separable preduals. This remarkable classification built on important earlier work by several mathematician was mostly accomplished by A. Connes with the final case settled by U. Haagerup. Part of Connes and Haagerup's  classification is the proof that several natural classes of factors (the injective, semi-discrete, approximately finite dimensional) coincide. This classification has very close connections and analogies with the theory of nuclear C*-algebras. Nuclear C*-algebras form an important class of C*-algebras since the important work of Choi and Effros, Connes, Effros and Lance, and Haagerup (see for example \cite{ChoiEffrosDuke}, \cite{ChoiEffros}. \cite{ConnesInjective}, \cite{EffrosLance}, \cite{HaagerupNUC} ) and are the main object of study in the Elliott classification program. This class includes important examples of C*-algebras coming from group representation theory, dynamical systems and mathematical physics.

The main result of this paper is a characterization of properly infinite injective von Neumann algebras and of nuclear C*-algebras by using a uniqueness theorem, based on generalizations of Voiculescu's famous Weyl-von Neumann theorem. 

Before stating Hadwin's generalization of 
Voiculescu's theorem, proved in \cite{HadwinTAMS}, let us recall that if $\h$ 
is a Hilbert space and $T\in \mathbb{B}(\h)$, then the rank of $T$, denoted by 
rank$(T)$ is the dimension of the closure of the range of $T$.

\begin{thm}
[\cite{HadwinTAMS}] Let $\A$ be a C*-algebra and $\h$ be a Hilbert space. 
Let
$\phi, \psi$ be two *-homomorphisms from $\A$ to $\mathbb{B}(\h)$. 

Then a necessary and sufficient condition for $\phi$ and $\psi$ to be approximately unitarily equivalent is that
$$\text{rank}\, (\phi(a)) = \text{rank} \,(\psi(a)),\;\; \text{ for all}\;  a \in \A\,.$$  
\end{thm}

Our first goal in this paper is to present versions of Theorem 1.1, where either $B(H)$ is replaced by a semidiscrete von Neumann algebra or the C*-algebra $A$ is nuclear and weaker notions of approximate unitary equivalence are used. 

Let us introduce these notions.

\begin{df}
Let $\A, \B$ be C*-algebras, with $\B$ unital, and let  
$\phi$ and $ \psi$ be two *-homomorphisms from $\A$ to $\B$. Then 

\item{(a)} $\phi$ and $ \psi$ are said to be approximately unitarily equivalent
 if there exists a net $\{u_{\alpha}\}$ of unitaries in $\B$ such that for 
all $a \in \A$
$$u_{\alpha} \phi(a) u_{\alpha}^* \rightarrow \psi(a)\quad 
\text{in the norm topology on}\;\; \B\,.$$

\item{(b)} $\phi$ and $ \psi$ are said to be weakly approximately unitarily 
equivalent if there exist two nets $\{ u_{\alpha} \}$ and $\{ v_{\beta} \}$ 
of unitaries
in $\B$ such that for all $a \in \A$,
$$u_{\alpha} \phi(a) u_{\alpha}^* \rightarrow \psi(a)$$
and
$$v_{\beta} \psi(a) v_{\beta}^* \rightarrow \phi(a)$$
in the relative weak topology on $\B$ (ie., the $\sigma(\B,\,\B^*)$-topology). 

\item{(c)} If $\B$ is a von Neumann algebra $\Mul$
, then $\phi$ and $ \psi$ are weak* approximately unitarily equivalent if there exist two nets $\{ u_{\alpha} \}$ and $\{ v_{\beta} \}$ of unitaries
in $\Mul$
 such that for all $a \in \A$,
$$u_{\alpha} \phi(a) u_{\alpha}^* \rightarrow \psi(a)$$
and
$$v_{\beta} \psi(a) v_{\beta}^* \rightarrow \phi(a)$$
in the weak*  topology on $\Mul$
 (i.e., the $\sigma(\Mul,\,\Mul_*)$-topology). 
\end{df}

These notions of approximate unitarily equivalence have been previously defined and studied, with $\B = \mathbb{B}(\h)$ (see  \cite{ArvesonDuke}, 
\cite{Davidson} and \cite{HadwinIND}) and in \cite{DingHadwin} for a 
von Neumann algebra $\B$. (The relationship between our work and \cite{DingHadwin} is \label{prop:II_1uniqueness} for a nuclear C*-algebra.) 

Before stating our first result, let us recall the following notion of rank introduced by Hadwin in  \cite{HadwinTAMS}:
two *-homomorphisms  $\phi$ and $ \psi$ from a C*-algebra $\A$
to a von Neumann algebra $\Mul$
have the same W*-rank if for each positive element $a \in \A$, the support projections of $\phi(a)$ and $\psi(a)$ are Murray-von Neumann equivalent in $\Mul$. 

Our first result is then:

\begin{thm} 
Let $\A$ be a C*-algebra and $\Mul$
 be a von Neumann algebra, and $\phi$ and $ \psi$ be two *-homomorphisms 
from $\A$ to $\Mul$.

If either $\A$ is nuclear or $\Mul$
 is semidiscrete, then for $\phi$ and $ \psi$ to be weak* approximately unitarily equivalent, it is necessary and sufficient that they have the same W*-rank.
\end{thm}

The proof of Theorem 1.2 is obtained by first considering the case of factors with separable preduals, then of general von Neumann algebras also with separable preduals, using reduction theory, before proving the general case. 

Theorem 1.2 and its converse for a properly infinite von Neumann algebra give a new characterization of semidiscreteness. Indeed we have:

\begin{thm} Let $\Mul$ be a
properly infinite von Neumann algebra.
Then the following statements are equivalent:

\begin{enumerate}
\item $\Mul$ is semidiscrete.
\item Let $\A$ be a C*-algebra and $\phi, \psi : \A \rightarrow \Mul$
two *-homomorphisms. Then $\textrm{W*-rank}(\phi) = \textrm{W*-rank}(\psi)$ if and only if $\phi$ and $\psi$ are
weak* approximately unitarily equivalent.
\end{enumerate}
\label{thm:InjectivityCharacterization}
\end{thm}

As in Theorem 1.3, our new characterization for nuclear C*-algebras uses two uniqueness properties for C*-algebras. To state it we need first (see  
Definition \ref{df:W*-rank}) to introduce a notion of C*-rank for pairs 
of *-homomorphisms from a C*-algebra $\A$ to a unital one $\B$, based on 
studies of the Cuntz semigroup associated to a C*-algebra.

\begin{thm}  Let $\A$ be a separable C*-algebra.
Then the following are equivalent:
\begin{enumerate}
\item $\A$ is nuclear.
\item $\A$ has the weak* uniqueness property (ie. for every von Neumann algebra $ \Mul $, for all pairs of 
*-homomorphisms $\phi$ and $ \psi$ from $\A$ to $\Mul$,  $\phi$ and $\psi$ have the same W*-rank if and only if they are weak* approximatively unitarily equivalent).
\item $\A$ has the weak uniqueness property (ie. for every unital C*-algebra 
$\B$, for all pairs of 
*-homomorphisms $\phi$ and $ \psi$ from $\A$ to $\B$,  $\phi$ and $ \psi$ have the same C*-rank if and only if they are weakly approximatively unitarily equivalent).

\end{enumerate}
\label{thm:nuclearitycharacterization}
\end{thm}

In C*-algebra theory, uniqueness theorems play an important role both in extension theory (see for example   \cite{ArvesonDuke},  \cite{Linbook}, \cite{EK} ) and in Elliott classification program for nuclear C*-algebras, but there the approximately unitarily equivalence in norm is used. 

\bigskip
\noindent
{\bf Notations:} Let $\Mul$ be a von Neumann algebra and $\Mul_*$ be its predual. 

\noindent
-  For all $\rho \in \Mul_*$, let $||\cdot||_{w^*,\rho} $ denote the semi-norm defined for $x \in \Mul$ by
$$  ||x||_{w^*,\rho} = |\rho(x)| \;.$$
The weak* topology (or  $\sigma(\Mul, \Mul_*)$-topology), also called the $\sigma$-weak topology,  is the topology induced by the semi-norms $||\cdot||_{w^*,\rho}\,,\; \rho \in \Mul_*$.

\noindent
-  For all $\phi \in \Mul_*^+$, let $||\cdot||_{\phi} $ (resp. $||\cdot||_{\phi}^\sharp $) denote the semi-norm defined for $x \in \Mul$ by
$$  ||x||_{\phi} = \sqrt{\phi(x^*x)}\quad\text(resp.\;\; ||x||_{\phi}^\sharp = \sqrt{\phi(x^*x +xx^*)}\; )$$
The strong topology (or  $s(\Mul, \Mul_*)$-topology) is the topology induced by the semi-norms $||\cdot||_{\phi}\,,\; \phi \in \Mul_*^+$ and the 
strong* topology (or  $s^*(\Mul, \Mul_*)$-topology) the topology induced by the semi-norms $||\cdot||_{\phi}^\sharp\,,\; \phi \in \Mul_*^+$.

It is well-known (see for example \cite{TakesakiBook}, II. 4. 10) that on the unitary group $U(\Mul)$ all the above topologies coincide.

\section{Uniqueness for finite von Neumann algebra codomains}

\begin{prop}
Let $\Mul$ be a countably decomposable finite von Neumann algebra. Let $\A$ be a C*-algebra. 

Suppose that either $\A$ is nuclear or $\Mul$ is injective.

Let
$\phi, \psi : \A \rightarrow \Mul$ be two *-homomorphisms.
Then
$$\tau \circ \phi = \tau \circ \psi$$
for every normal tracial state $\tau$ on $\Mul$
if and only if
there exist two nets $\{ u_{\alpha} \}$ and $\{ v_{\beta} \}$ of unitaries
in $\Mul$ such that for all $a \in \A$,
$$u_{\alpha} \phi(a) u_{\alpha}^* \rightarrow \psi(a)$$
and
$$v_{\beta} \psi(a) v_{\beta}^* \rightarrow \phi(a)$$
where the convergence is in the $\sigma$-strong* topology on $\Mul$.
\label{prop:II_1uniqueness}
\end{prop}

\begin{proof}
   The if direction is clear.  So it suffices to prove the only if direction.

   Note that the condition $\tau \circ \phi = \tau \circ \psi$ for
every $\tau$ implies that
$\phi$ and $\psi$ have the same kernel.
Hence, we may assume that both $\phi$ and $\psi$ are injective,
$\A \subset \Mul$, and
$\phi$ is the inclusion of $\A$ into $\Mul$.

   Let $\epsilon > 0$ and let $\F = \{ a_1, a_2, ...,
a_n \} \subset \A$ be a finite subset.  Let $\G$ be a finite collection of normal tracial states
on $\Mul$.  As $\A''$ is hyperfinite, there exists a finite dimensional von Neumann subalgebra $F
\subset \A''$ and $x_1, x_2, ..., x_n \in F$ such that $\| x_i - a_i \|_{\tau} < \epsilon$ for $1
\leq i \leq n$, and for  every $\tau \in \G$. (See, for example, \cite{ChoiEffros},
\cite{ConnesInjective}, \cite{ElliottNonseparableAFD}.)

   Let $\overline{\psi}$ be the extension of $\psi$ to $\A''$.
We have $\tau \circ \overline{\psi} = \tau$ for
every normal tracial state $\tau$ on $\Mul$.
Hence,
$$\| \psi(a_i) - \overline{\psi}(x_i) \|_{\tau} = \| \overline{\psi}(a_i -
x_i) \|_{\tau} = \| x_i - a_i \|_{\tau}$$
for every $\tau \in \G$ and
for $1 \leq i \leq n$.

Since $F$ and $\overline{\psi}(F)$ are finite dimensional
von Neumann subalgebras
of $\Mul$ and $\tau \circ \overline{\psi} = \tau$ for every normal
tracial state $\tau$ on $\Mul$, there exists
a unitary $u \in \Mul$ such that
$$\overline{\psi} |_F = Ad(u)$$

    Hence, for $1 \leq i \leq n$, and for every
$\tau \in \G$,
\begin{eqnarray*}
& & \| \psi(a_i ) - u a_i u^* \|_{\tau} \\
& \leq & \| \psi(a_i) - \overline{\psi}(x_i) \|_{\tau} +
\| \overline{\psi}(x_i) - u x_i u^* \|_{\tau} + \| u (x_i - a_i)
u^* \|_{\tau} \\
& < & \epsilon + 0 + \epsilon \\
& = & 2 \epsilon
\end{eqnarray*}
as required.
\end{proof}

\section{Uniqueness for properly infinite von Neumann factor codomains}

    We will need the following excision of pure states result (a generalization of
Glimm's Lemma), due to
Akemann, Anderson and Pedersen,  whose
proof can be found in \cite{AAP} Proposition 2.2.
(See also \cite{EK} Lemma 8 or
\cite{Linbook} Lemma 5.3.2 for a short proof of special case.):

\begin{lem}  Let $\A$ be a C*-algebra.  Let $\rho$
be a pure state on $\A$.  Then there exists a net $\{ a_{\alpha} \}$ of positive elements in $\A$
with $\| a_{\alpha} \| = \rho(a_{\alpha}) = 1$ for all $\alpha$, such that
\[
lim_{ \alpha } \| a_{\alpha} (a - \rho(a)) a_{\alpha} \| = 0.
\]
for all $a \in \A$.
\label{lem:excision}
\end{lem}

\begin{lem}
Let $\Mul$ be a properly infinite von Neumann algebra and let $\A$ be a C*-algebra with a faithful
state. Let $\sigma : \A \rightarrow \M_n (\mathbb{C})$ and $\eta : \M_n( \mathbb{C}) \rightarrow
\Mul$ be completely positive contractive maps. Let $\psi : \A \rightarrow \Mul$ be the completely
positive contractive map given by
\[
\psi =_{df} \eta \circ \sigma
\]
Then $\psi$ can be approximated in the pointwise-norm operator topology
by finite sums of maps of the form
\[
a \mapsto m \rho_n(a_0^* a a_0 ) m^*
\]
\noindent where $\rho$ is a pure state on $\A$, $\rho_n  =_{df} \rho \otimes id_{\M_n(\mathbb{C})}
: \M_n(\A) \rightarrow \M_n ( \mathbb{C})$ is the natural map induced by $\rho$, $m$ is a row
matrix in $\Mul^n$ and $a_0$ is a row matrix in $\A^n$. \label{lem:semidiscretemaps}
\end{lem}

\begin{proof}

 Firstly, by \cite{AnanHave} Lemma 4.4 and the Krein--Milman Theorem,
the map $\sigma$ can be approximated on finite sets (i.e., in the pointwise norm topology) by
finite sums of maps of the form
$$ \A \rightarrow \M_n ( \mathbb{C}) : a \mapsto \rho_n(a_0^* a a_0)$$
where $a_0$ is a row matrix over $\A$ with length $n$ and where $\rho$ is a pure state on $\A$.

  Hence, to complete the proof, it suffices to prove
that there exists a row matrix $m$ over $\Mul$, with length $n$, such that $\nu$ can be expressed
in the form
$$\M_n \rightarrow \Mul : x \mapsto m x m^*.$$
But this follows from \cite{Haagerup} Proposition 2.1.

\end{proof}

The above lemma also works with the von Neumann algebra $\Mul$ replaced with a unital C*-algebra
which contains a unital copy of the Cuntz algebra $O_2$. The proof involves a variation on the
argument of \cite{Haagerup} Proposition 2.1.

   In order to make our arguments go through, we need to restrict
ourselves to the factor case.  We will eventually remove this condition.

   For a properly infinite von Neumann factor $\Mul$, we let
$\K_{\Mul}$ denote the Breuer ideal of $\Mul$, i.e., the C*-ideal of $\Mul$ generated by the
finite projections.  (Hence, if $\Mul$ is type III then $\K_{\Mul} = 0$.)

\begin{lem}  Let $\Mul$ be a countably decomposable properly infinite von Neumann factor.
Let $b_i \in \Mul$ for $1 \leq i \leq n$, and let $p', q' \in \Mul$ be infinite (i.e., $p', q'
\notin \K_{\Mul}$) projections.

Then there exist  infinite projections $p, q \in \Mul$ such that $p \leq p'$, $q \leq q'$ and $p
b_i q = 0$ for $1 \leq i \leq n$. \label{lem:NullCrossTerms}
\end{lem}

\begin{proof}[Sketch of proof:]
By induction, it suffices to prove this for $n = 1$.

Let $p'' \in \Mul$ be the left support projection of $p' b_1 q'$.  Then $p'' \leq p'$.

If $p' - p''$ is an infinite projection, then take $p =_{df} p' - p''$ and $q =_{df} q'$.

Suppose that $p' - p''$ is not an infinite projection.  
Then $p''$ is infinite and hence, a
properly infinite projection.  Hence, let $p_1, p_2 \in \Mul$ be pairwise orthogonal projections
such that $p_1 \sim p_2 \sim p''$ and $p'' = p_1 + p_2$.   Let 
$q_1 \in \Mul$ be the right support
projection of $p_1 b_1 q'$.
Then $q_2 =_{df} q' - q_1$  is an infinite projection such that 
$p_1 b_1 q_2 = 0$.  So take $p =_{df} p_1$ and $q =_{df} q_2$.  
(Clearly, $p_1 b_1 q_2 = 0$.  
So it suffices to show that $q_2$ is an infinite projection.
Let $p'_2$ be the left support projection of $p' b_1 q_2$.
Then since $q_1 \oplus q_2 = q'$,
$p_1 \vee p'_2 = p'' = p_1 \oplus p_2$.
Hence, $p_1 \vee p'_2 - p_1 = p_2$ which is an infinite projection.
But, by \cite{TakesakiBook} Proposition V.1.6, 
$(p_1 \vee p'_2) - p_1 \sim p'_2 - (p_1 \wedge p'_2)$. 
Hence, $p'_2$ is an infinite projection.  Hence, $q_2$ is an infinite
projection.)   
\end{proof}

Before continuing, we recall some notation used, for example, in the study of the Cuntz semigroup
 (e.g., \cite{OrRorThi}). For $\epsilon
> 0$, let $(t - \epsilon)_+ : (-\infty, \infty) \rightarrow [0, \infty)$ be the function which is
given by
\[
(t - \epsilon)_+
=
\begin{cases}
t - \epsilon & \makebox{  if  } t \geq \epsilon \\
0   & \makebox{  otherwise  }
\end{cases}
\]
Let $\A$ be a C*-algebra.  For a self-adjoint element $a \in \A$ and for $\epsilon > 0$, let $(a -
\epsilon)_+ \in \A$ be the positive element gotten by applying $(t - \epsilon)_+$ and the
continuous functional calculus to $a$.  For a positive element $c \in \A$, let $Her(c) \subseteq
\A$ denote the hereditary C*-subalgebra generated by $c$ (i.e., the smallest hereditary
C*-subalgebra of $\A$ that contains $c$.)

\begin{lem}
Let $\Mul$ be a countably decomposable properly infinite von Neumann factor and let $\A$ be a
$C^*$-subalgebra of $\Mul$ with a unit. Let $p \in \K_{\Mul}$ be a finite projection. Suppose that
$\sigma : \A \rightarrow \M_n ( \mathbb{C})$ and $\eta : \M_n (\mathbb{C}) \rightarrow \Mul$ are
two completely positive maps with the following properties:
\begin{enumerate}
\item[(a)] If $\psi =_{df} \eta \circ \sigma$
then $\psi(1_{\A})$ is a projection
in $\Mul$
\item[(b)] $\sigma |_{\A \cap \K_{\Mul}} =
\psi |_{\A \cap \K_{\Mul}} = 0$
\end{enumerate}

Then $\psi$ can be approximated in the pointwise-norm topology
by maps of the form
\[
a \mapsto v^* a v
\]
where $v$ is a partial isometry in $\Mul$ such that
$$p v = 0$$
\label{lem:Preproperlyinfinitenullideal}
\end{lem}

\begin{proof}
   If $\A \subset \K_{\Mul}$ then we can take $v = 0$.  Hence,
we may assume that $\A$ is not a subset of $\K_{\Mul}$.   

    First, let $\epsilon > 0$ be given and
let $\F$ be a finite subset of $\A$.
We may assume that $1_{\A} \in \F$.
Also, we may assume that $p \in \A$ and
$ap, pa, pap \in \F$ for all $a \in \F$.
We will approximate $\psi$ on $\F$ in the norm topology.
For simplicity,
we may assume that the elements of $\F$ all have
norm less than or equal to one.
Let $\delta > 0$ be arbitrary.

By Lemma \ref{lem:semidiscretemaps},
let
$\rho^1, \rho^2,..., \rho^k$ be a finite set of pure states on
$\A$ with $\rho^i |_{\A \cap \K_{\Mul}} = 0$ for $1 \leq i \leq k$, let
$m_1, m_2, ..., m_k$ be a set of row matrices in
$\Mul^n$ and let
$a_1, ..., a_k$ be a set of row matrices in $\A^n$ such that
on $\F$,
$\psi$ is within $\delta$ of the map
\[
a \mapsto \sum_{i = 1}^k m_i {\rho^{i}}_n ( a_i^* a a_i ) m_i^*.
\]

By Lemma \ref{lem:excision}, for each
$i$, let $c_i$ be a positive element of $\A$ with $\| c_i \| = 1$ and
$\rho^i(c_i) = 1$  such
that
$$diag(c_i, c_i, ..., c_i) {\rho^i}_n
(a_i^*a a_i) diag(c_i, c_i, ..., c_i)$$
is within $\frac{\delta}{2k ( \| m_i \|^2 + 1)}$ of
$$diag(c_i, c_i, ..., c_i) a_i^* a a_i diag(c_i, c_i, ..., c_i)$$
for
every $a \in \F$.  Here,
$diag(c_i, c_i, ..., c_i)$ is the element of
$\M_n( \A )$ with $c_i$'s in the diagonal and zeroes everywhere else.

    Note that for $1 \leq i \leq k$, since $\rho^i |_{\A \cap \K_{\Mul}} =0$
and $\rho^i(c_i) = 1$, $c_i$ is a full element of $\Mul$.
Hence, since $\Mul$ is a properly infinite factor,
for each $i$, let
$x_i$ be an element of $\Mul$ with norm less than $5/4$ such that
for every $a \in \F$,
\begin{enumerate}
\item[(a)] $x_i c_i^2 x_i^* = 1_{\Mul}$ for each $i$, and
\item[(b)] $x_i c_i a_i^* a a_j c_j x_j^* = 0$ for
$i \neq j$.
\end{enumerate}
(Indeed, since $\| c_i \| = \rho^i(c_i) = 1$,  for small enough $\epsilon>0$, the element $(c_i-1+\epsilon)_+$ is not compact and hence is full. Therefore, $Her( (c_i -1 + \epsilon)_+)$ contains an infinite projection, say, $p'_i \in \mathcal M$  for each $1 \leq i \leq n$. Then, by repeatedly applying Lemma
\ref{lem:NullCrossTerms}, one gets infinite subprojections $p_i \leq p'_i$ in $\mathcal M$ such that $p_i
c_i a_i^* a a_j c_j p_j^* = 0$ for $i \neq j$.  Note that $p_i \sim 1_{\mathcal M}$ ($1 \leq i \leq n$).
Then take $x_i =_{df} y_i p_i$ for appropriate $y_i \in \mathcal M$.  Since $p_i \in Her( (c_i -1 +
\epsilon)_+)$, $y_i$ can be chosen so that $\| x_i \| < 5/4$ ($1 \leq
i \leq n$).)

%(Here is a rough sketch of the argument to get the $x_i$:  Firstly, use that $\| c_i \| =
%\rho^i(c_i) = 1$ to prove that for small $\epsilon > 0$, $Her( (c_i -1 + \epsilon)_+)$ contains an
%infinite projection, say, $p'_i \in \Mul$ ($1 \leq i \leq n$).  Next, repeatedly apply Lemma
%\ref{lem:NullCrossTerms} to get infinite subprojections $p_i \leq p'_i$ in $\Mul$ such that $p_i
%c_i a_i^* a a_j c_j p_j^* = 0$ for $i \neq j$.  Note that $p_i \sim 1_{\Mul}$ ($1 \leq i \leq n$).
%Then take $x_i =_{df} y_i p_i$ for appropriate $y_i \in \Mul$.  Since $p_i \in Her( (c_i -1 +
%\epsilon)_+)$, for small enough $\epsilon$, $y_i$ can be chosen so that $\| x_i \| < 5/4$ ($1 \leq
%i \leq n$).)

    Take $x^* =_{df} \sum_{i=1}^k m_i x_i c_i a_i^*$.
Then on $\F$, $\psi$ is within $2 \delta$ of the map $a \mapsto x^* a x$. Since $ap, pa, pap \in
\F$ for all $a \in \F$, and since $\psi(ap) = \psi(pa) = \psi(pap) = 0$ (since $p \in \K_{\Mul}$)
for all $a \in \F$, we have that
$$\| x^* pa x \|, \| x^* ap x \|, \| x^* pap x \| < 2\delta$$
for all $a \in \F$.
Now take $y =_{df} (1 - p) x \in \Mul$.
Then for all $a \in \F$,
$$\| y^* a y - \psi(a) \|
\leq  \| x^* a x - \psi(a) \| + \| x^* a p x \| + \| x^* p a x \|
+ \| x^* p a p x \| < 8 \delta.$$
Also, by hypothesis, $\psi(1_{\A})$ is a projection and
$\| y^* 1_{\A} y - \psi(1_{\A}) \| < \delta$.

Hence, since $\delta$ was arbitrary, if we chose $\delta$ to be
small enough then we can find a partial isometry
$v \in \Mul$ such
that $pv = 0$ and on $\F$, $\psi$ is within $\epsilon$ of the map
$$a \mapsto v^* a v,$$
as required.
\end{proof}

\begin{cor} Let $\Mul$ be a countably decomposable properly infinite
von Neumann factor
and let $\A$ be a  C*-subalgebra of $\Mul$.
Let $p \in \K_{\Mul}$ be a finite projection, and suppose that
$\psi : \A \rightarrow \Mul$ is a *-homomorphism such that
$$\psi |_{\A \cap \K_{\Mul}} = 0$$

Then if $\A$ is nuclear (resp. $\Mul$ is semidiscrete) then $\psi$ can be approximated in the
pointwise-norm (resp. pointwise-$\sigma$-strong*) topology by maps of the form $a \mapsto v^* a v$
where $v$ is a partial isometry in $\Mul$
such that $pv = 0$. \label{cor:properlyinfinitenullideal}
\end{cor}

\begin{proof}

   We may assume that $\A$ is unital.  For otherwise, we can replace
$\A$ with $\A + \mathbb{C} 1_{\overline{\A}^{strong}}$ and extend $\psi$ to $\tilde{\psi}$ where
$\tilde{\psi} |_{\A} = \psi$ and $\tilde{\psi}(1_{\overline{\A}^{strong}}) =
1_{\overline{\psi(\A)}^{strong}}$. One can check that $\tilde{\psi}( (\A +
\mathbb{C}1_{\overline{\A}^{strong}}) \cap \K_{\Mul}) = 0$. (Indeed, say that $a \in \A$ is a
self-adjoint element with $\| a \| \leq 1$ such that $1_{\overline{\A}^{strong}}  - a \in
\K_{\Mul}$. Choose an approximate unit $\{ b_{\alpha} \}$ for $\A$ such that $|a| \leq b_{\alpha}$
for all $\alpha$. Then $b_{\alpha} - a \rightarrow 1_{\overline{\A}^{strong}} - a$ and
$\psi(b_{\alpha} - a) \rightarrow 1_{\overline{\psi(\A)}^{strong}} - \psi(a)$ in the strong
topology. But $0 \leq b_{\alpha} - a \leq 1_{\overline{\A}^{strong}} - a \in \K_{\Mul}$ for all
$\alpha$. Hence, $\psi(b_{\alpha} - a) = 0$ for all $\alpha$. Hence,
$\tilde{\psi}(1_{\overline{\A}^{strong}} - a) = 1_{\overline{\psi(\A)}^{strong}} - \psi(a) = 0$.)

The corollary then follows from Lemma \ref{lem:Preproperlyinfinitenullideal} and the definitions
of nuclearity and semidiscreteness.
\end{proof}

% Towards introducing nontrivial K-theory (in the nonstable sense)
%into our uniqueness theorem for the general
%properly infinite case, we introduce
The following notion of rank %which is natural in the von Neumann algebra context.\\
was introduced by Hadwin (\cite{HadwinTAMS}).

\begin{df}  Let $\A$ be a C*-algebra and let $\Mul$ be a von Neumann algebra.
Let $\phi, \psi : \A \rightarrow \Mul$ be two *-homomorphisms.

  Then we say that $\phi$ and $\psi$ have the same \emph{W*-rank} (and
write ``$\textrm{W*-rank}(\phi) = \textrm{W*-rank}(\psi)$") if for every positive element $a \in \A$, the support
projections of $\phi(a)$ and $\psi(a)$ are Murray-von Neumann equivalent in $\Mul$.
\label{df:rank}
\end{df}

With notation as in Definition \ref{df:rank},  by \cite{BlackadarGeneralBook} Theorem III.2.5.7,
we have that in the case of a finite von Neumann algebra $\Mul$,  $\textrm{W*-rank}(\phi) = \textrm{W*-rank}(\psi)$
if and only if $\tau \circ \phi= \tau \circ \psi$ for every normal tracial state $\tau$ on $\Mul$
if and only if $T \circ \phi = T \circ \psi$ where $T$ is the unique centre-valued trace on
$\Mul$. (Compare with Proposition \ref{prop:II_1uniqueness}.)

For the convenience of the reader, we recall some notation (introduced in the introduction). For a
von Neumann algebra $\Mul$ and a normal linear functional $\rho \in \Mul_*$, recall that $\| .
\|_{w*, \rho}$ is the seminorm on $\Mul$ given by $\| x \|_{w*, \rho} =_{df} |\rho(x)|$ for all $x
\in \Mul$.

\begin{lem}  Let $\Mul$ be a countably decomposable
properly infinite von Neumann factor and
let $\A$ be a C*-algebra.
Suppose that either $\A$ is nuclear or $\Mul$ is semidiscrete.

Then if
$$\phi, \psi : \A \rightarrow \Mul$$
are two *-homomorphisms such that
$$\textrm{W*-rank}(\phi) = \textrm{W*-rank}(\psi)$$

then there exists a net $\{ v_{\alpha} \}$ of partial isometries in $\Mul$ such that for all $a
\in \A$,

$$v_{\alpha}^* \psi(a) v_{\alpha} \rightarrow \phi(a)$$
in the weak* topology. \label{lem:ProperlyInfiniteFactorCodomain}
\end{lem}

\begin{proof}

   By \cite{HadwinTAMS}, we may assume that $\Mul$ is a continuous properly infinite factor.
If $\Mul$ is a type $III$ factor then this follows from Corollary
\ref{cor:properlyinfinitenullideal}. Hence, we may assume that $\Mul$ is a type $II_{\infty}$
factor.

   That $\textrm{W*-rank}(\phi) = \textrm{W*-rank}(\psi)$ implies that $ker(\phi) = ker(\psi)$.
Hence, replacing $\A$ with $\A / ker(\psi)$ if necessary, we may assume that $\phi$ and $\psi$ are
injective.  We may further assume that $\A$ is a C*-subalgebra of $\Mul$ and $\psi : \A
\rightarrow \Mul$ is the natural inclusion map.

   Let $\epsilon > 0$ be given and let $\G \subset \Mul_*$ be a finite
collection of normal states.  Let $\F \subset \A$ be a finite set
of elements.  We may assume that the elements of $\F$ have norm less
than or equal to one.

    Let $\delta >0$ be arbitrary.

   Let $P, Q \in \Mul$ be the projections that are given by
 $P =_{df} 1_{\overline{\A \cap \K_{\Mul}}^{strong}}$
 and $Q =_{df} 1_{\overline{\phi(\A) \cap \K_{\Mul}}^{strong}}$ respectively. $P$ is the strong limit of an
approximate unit for $\A \cap \K_{\Mul}$ which quasicentralizes $\A$ (see \cite{ArvesonDuke}
Theorem 1). Hence, choose a positive element $e \in \A \cap \K_{\Mul}$ with norm one such that the
following statements are true:
\begin{equation} \label{equ:almostorthogonalsummand} \end{equation}
\begin{enumerate}
\item[i.]  If $p$ is the support projection of $e$ then
$p \in \K_{\Mul}$.
\item[ii.] The elements $e a, ae, eae$ are all within $\delta$ of each other,
for all $a \in \F$.
\item[iii.]  There is a projection $r \in \K_{\Mul}$ with
$r \leq e$ such that $|\rho(P - r)|, |\rho(Q - \phi(r))| < \delta$, for
all $\rho \in \G$.
\end{enumerate}

Since $Q$ is the (strong) limit of an approximate unit for $\phi(\A) \cap \K_{\Mul}$ that
quasicentralizes $\phi(\A)$, we have that $Q \phi(a) = \phi(a) Q$ for every $a \in \A$.
Hence, we
have a *-homomorphism
\[
\A \rightarrow \Mul : a \mapsto (1 - Q) \phi(a) (1 - Q).
\]
Since $\textrm{W*-rank}(\phi) = \textrm{W*-rank}(\psi)$, $\phi(\A \cap \K_{\Mul}) = \phi(\A) \cap \K_{\Mul}$.
(Indeed, for every positive $a \in \A$, $a \in \K_{\Mul}$ if and only if $(a - \mu)_+ \in
\K_{\Mul}$ for all $\mu > 0$ if and only if the support projection of $(a - \mu)_+$ is in
$\K_{\Mul}$ for all $\mu > 0$ if and only if the support projection of $(\phi(a) - \mu)_+ =
\phi((a - \mu)_+)$ is in $\K_{\Mul}$ for all $\mu
> 0$ (since $\textrm{W*-rank}(\phi) = \textrm{W*-rank}(\psi)$) if and only if $(\phi(a) - \mu)_+ \in \K_{\Mul}$ for all
$\mu > 0$ if and only if $\phi(a) \in \K_{\Mul}$.)  Hence, the above map annihilates $\A \cap
\K_{\Mul}$. Therefore, by Corollary \ref{cor:properlyinfinitenullideal}, let $v \in \Mul$ be a
partial isometry such that $v v^* = 1 -Q$, $vp = 0$ and
\begin{equation}
\| v a v^* - (1 - Q) \phi(a) (1 - Q)  \|_{w*, \rho}< \delta \label{equ:1-Qpiece}
\end{equation}
for all $a \in \F$ and all $\rho \in \G$.\\

    Both $p \Mul p$ and $\phi(p) \Mul \phi(p)$ are
type $II_1$ factors.  Hence, by
Proposition \ref{prop:II_1uniqueness} (taking $e \A e$ as the domain
algebra and $\phi |_{e \A e}$, $\psi |_{e \A e}$ as the maps),
there exists a partial isometry $w$ with
$p = w^* w$ and $\phi(p) = w w^*$ such that
\begin{equation}
\| w e a e w^* - \phi(e a e) \|_{w*, \rho} < \delta \label{equ:Qpiece}
\end{equation}
for all $a \in \F$ and all $\rho \in \G$.

   From (\ref{equ:almostorthogonalsummand}),  (\ref{equ:1-Qpiece}) and
(\ref{equ:Qpiece}), for all $a \in \F$ and all $\rho \in \G$,
\begin{eqnarray*}
 & & \|(v + w) a (v + w) - \phi(a) \|_{w*, \rho} \\
& < & 4 \delta + \|(v + w)(e ae + (1 - e) a (1 - e)) (v^* + w^*) - \phi(e a e + (1 - e) a (1
- e)) \|_{w*, \rho} \\
& < & 8 \delta  + \| w e ae w^* - \phi(eae) \|_{w*, \rho}
+ \| v a v^* - (1 - Q) \phi(a) (1 - Q) \|_{w*, \rho} \\
& < & 10 \delta.
\end{eqnarray*}

   Since $\delta$ is arbitrary, if we chose $\delta = \epsilon/10$ then
we would have that $\phi$ is $\epsilon$-approximately inner over $\F$
and with respect to $\G$.
\end{proof}

The above lemma generalizes the results of \cite{Voiculescu}, \cite{HadwinTAMS} and
\cite{ArvesonDuke} which proved the case of type I codomains.  In fact, in these papers, the
convergence (for the approximate unitary equivalence)is stronger (in the norm topology).

The next result seems standard, but we did not find an exact reference.

\begin{lem}  Let $\Mul$ be a countably decomposable properly infinite von Neumann algebra
and let $\F \subset \Mul$ be a finite subset.
Let $\G \subset \Mul_*$ be a finite collection of normal states.
Then for every partial isometry $v \in \Mul$, for every
$\epsilon > 0$, there exists a unitary $u \in \Mul$ such that
\[
\|  v a v^* - u a u^* \|_{w*, \rho} < \epsilon
\]
for all $a \in \F$ and for all $\rho \in \G$.
\label{lem:partialisometrytounitary}
\end{lem}

\begin{proof}[Sketch of Proof]
  We may assume that the elements of $\F$ all have norm less than or equal
to one.

   Let $\delta > 0$ be arbitrary.

   Since $\Mul$ is properly infinite, we can find a projection
$p \in \Mul$ with $p \leq v^* v$ such that $|\rho( v^*a v  - v^* p a p v )| < \delta$ for all $a
\in \F$ and all $\rho \in \G$, and $1 - p \sim 1 \sim 1 - vpv^*$.

   Since $1 - p$ and $1 - vp v^*$ is properly infinite, we can
find a partial isometry $w \in \Mul$ with initial projection $1 - p$
and range projection $1 - v pv^*$ such that
for all $a \in \F$ and all $\rho \in \G$,
\[
|\rho( v p a w^* )| + | \rho(w a p v^*) | + | \rho(w a w^* )| < \delta.
\]

   Let $u \in \Mul$ be the unitary given by
$$u =_{df} u p + w.$$

   Then for all $a \in \F$ and all $\rho \in \G$,
$$|\rho( v a v^* - u a u^* )| < 2 \delta.$$
Since  $\delta$ is arbitrary, if we had chosen $\delta = \epsilon/2$ then the proof would be
complete.
\end{proof}

\begin{cor}  Let $\Mul$ be a countably decomposable
properly infinite von Neumann factor and let $\A$ be a C*-algebra.
Suppose that either $\A$ is nuclear or $\Mul$ is injective.

Let $\phi, \psi : \A \rightarrow \Mul$ be two *-homomorphisms.

Then $\textrm{W*-rank}(\phi)= \textrm{W*-rank}(\psi)$ if and only if $\phi$ and $\psi$ are weak* approximately
unitarily equivalent. \label{cor:properlyinfinitefactoruniqueness}
\end{cor}

\begin{proof}
The ``only if" direction follows from Lemma \ref{lem:ProperlyInfiniteFactorCodomain}.

For the ``if" direction, if $\Mul$ is type III then use that two projections are Murray-von
Neumann equivalent if and only if they are both nonzero or both zero. If $\mathcal M$ is type II$_\infty$, then the equivalence of projections is determined by a normal semifinite trace $\tau$. Suppose, for contradiction, there were $a\in A^+$ such that the projection $\mathrm{supp}(\phi(a))$ is not Murray-von Neumann equivalent to the projection $\mathrm{supp}(\psi(a))$. Then, without loss of generality, one may assume that $\mathrm{supp}(\psi(a))$ is finite and $\tau(\mathrm{supp}(\phi(a)))>\tau(\mathrm{supp}(\psi(a)))$, and hence there is a continuous positive function $g$ such that $\tau(\phi(g(a)))>\tau(\psi (g(a)))$, which contradicts to the assumption that $\phi(g(a))$ is approximately conjugate to $\psi(g(a))$ in the weak* topology.

%If $\Mul$ is type
%$II_{\infty}$, then use that equivalence of projections is determined by the normal semifinite
%trace -- whose cutdowns by finite projections give positive normal linear functionals. 
%(Roughly
%speaking: Suppose, for contradiction, that $a \in \A$ is such that $supp(\phi(a))$ is an infinite
%projection but $supp(\psi(a))$ is a finite projection. By hypothesis, let $\{ u_{\alpha} \}$ be a
%net of unitaries such that for all $c \in \A$, $u_{\alpha} \psi(c) u_{\alpha}^* \rightarrow
%\phi(c)$ in the weak* topology.  Then there exist a continuous function $g$, a finite projection
%$p$, and a normal tracial state $\tau$ on $p \Mul p$ such that $\tau(p g(\phi(a))p) > \tau( p
%u_{\alpha}g(\psi(a)) u_{\alpha}^*p)$ for all $\alpha$. (Note that $0 \leq \tau( p
%u_{\alpha}g(\psi(a)) u_{\alpha}^*p) \leq \tau(g(\psi(a))) < \infty$ for all $\alpha$.)  This
%contradicts the fact that $u_{\alpha} \psi(g(a)) u_{\alpha}^* \rightarrow \phi(g(a))$. in the
%weak* topology.  If $supp(\phi(a))$ and $supp(\psi(a))$ are both finite projections, the proof
%that they are Murray--von Neumann equivalent is similar.)

\end{proof}

Note that  weak* approximate unitary equivalence is a relatively flexible notion. In Corollary
\ref{cor:properlyinfinitefactoruniqueness} we can have examples of unital $\A$ and weak*
approximately unitarily equivalent maps $\phi$ and $\psi$ where $\phi(1_{\A})  = 1_{\Mul}$ but
$\psi(1_{\A})$ is a proper subprojection of $1_{\Mul}$.  On the other hand, $\sigma$-strong*
approximate unitary equivalence is a more rigid notion. In particular, for convergence in the
$\sigma$-strong* topology, we need for both maps to be unital.

\begin{lem}  Let $\Mul$ be a countably decomposable von Neumann algebra and let $\A$ be a unital
C*-algebra.
Let $\phi, \psi : \A \rightarrow \Mul$ be two unital *-homomorphisms.

Then
$\phi$ and $\psi$ are weak* approximately unitarily equivalent if and only
if $\phi$ and $\psi$ are $\sigma$-strong* approximately unitarily
equivalent.

In particular, if $\{ u_{\alpha} \}$ is a net of unitaries in $\Mul$
such that
$u_{\alpha} \psi(a) u_{\alpha}^* \rightarrow \psi(a)$ weak* for all $a \in \A$,
then $u_{\alpha} \psi(a) u_{\alpha}^* \rightarrow \psi(a)$ $\sigma$-strong*
for all $a \in \A$.
\label{lem:EquivalentEquivalences}
\end{lem}

\begin{proof}
The ``if" direction is clear.

   The proof of the ``only if" direction follows from the fact that on the unitary group
of $\Mul$,
the weak* topology is the same as the $\sigma$-strong* topology. Also,
$\A$ is the (norm-) closed linear span its unitaries.
\end{proof}

\begin{cor}  Let $\Mul$ be a countably decomposable
properly infinite von Neumann factor and let $\A$ be a unital C*-algebra.
Suppose that either $\A$ is nuclear or $\Mul$ is injective.

Let $\phi, \psi : \A \rightarrow \Mul$ be two unital *-homomorphisms.

Then the following statements are equivalent:
\begin{enumerate}
\item $\textrm{W*-rank}(\phi) = \textrm{W*-rank}(\psi)$
\item $\phi$ and $\psi$ are weak* approximately  unitarily equivalent
\item $\phi$ and $\psi$ are $\sigma$-strong* approximately unitarily
equivalent.
\end{enumerate}
\end{cor}

\begin{proof}
   That (1) is equivalent to (2) follows from Corollary
\ref{cor:properlyinfinitefactoruniqueness}.
That (2) is equivalent to (3) follows from
Lemma \ref{lem:EquivalentEquivalences}.
\end{proof}

\section{Uniqueness for general von Neumann algebra codomains}

In this section, we will generalize the results of section 3 to a general von Neumann algebra $\Mul$ with separable predual by using its direct integral decomposition along its centre. Without loss of generality, we can assume that $\Mul$ acts on a separable Hilbert space $\h$. We refer the reader to \cite{KadisonVolume2}
Chapter 14 for notation and preliminary results.

The centre $Z(\Mul)$ of $\Mul$ is isomorphic to $L^{\infty}(T, \mu)$ where $(T, \mu)$ is a
(locally compact compete separable metric) measure space.  Let $\{ \h_t \}_{t \in T}$ and $\{
\Mul_t \}_{t \in T}$ be the corresponding direct integral decompositions of $\h$ and $\Mul$
respectively.

For a positive element $a$ in a von Neumann algebra $\N$,
let $supp(a)$ denote the support projection of $a$.
The first lemma is a standard computation.

\begin{lem}
Let $\Mul$ be a von Neumann algebra with separable predual and let $\Mul = \int^{\oplus}_T \Mul_t
 d \mu(t)$ be its central decomposition.
%$\{ \Mul_t \}_{t \in T}$ be a direct integral decomposition of
%$\Mul$, where $(T, \mu)$ is a standard measure space and $\Mul_t$ is
%a von Neumann algebra for all $t \in T$.

Suppose that $a, b \in \Mul$ are positive elements such that
$supp(a)$ is Murray-von Neumman equivalent to $supp(b)$ in $\Mul$.

Then $a(t), b(t) \in \Mul_t$ are positive elements and
$supp(a(t))$ is Murray-von Neumann equivalent to $supp(b(t))$, for
$\mu$-a.e. $t \in T$.
\label{lem:DirectIntegralEquivalence}
\end{lem}

\begin{prop}  Let $\Mul$ be a properly infinite von Neumann algebra
with separable predual, and let $\A$ be a separable C*-algebra.
Suppose that either $\A$ is nuclear or $\Mul$ is semidiscrete.

Let $\phi, \psi : \A \rightarrow \Mul$ be injective *-homomorphisms.

Then $\textrm{W*-rank}(\phi) = \textrm{W*-rank}(\psi)$ if and only if $\phi$ and $\psi$ are weak* approximately
unitarily equivalent.

\label{prop:separableproperlyinfinitecodomain}
\end{prop}

\begin{proof}

  To prove the sufficiency of the condition (i.e., the ``only if direction")
it is enough to prove the following statement:

Let $\epsilon > 0$ be given and let $\rho \in \Mul_*$ be a positive normal
linear functional.  Let $a_1, a_2, ..., a_n \in \A$ be elements
such that $\| a_i \| \leq 1$ for $1 \leq i \leq n$.
Then there exists a unitary $u \in \Mul$ such that
$$\| u \phi(a_i) u^* - \psi(a_i) \|_{w*, \rho} < \epsilon$$
for $1 \leq i \leq n$.

%We may assume that $\Mul$ sits on a separable infinite dimensional Hilbert space $\h$.

%Let $(T, \mu)$ be a complete standard measure space such that
%$T = \bigcup_{n=1}^{\infty} K_n$ where for all $n \geq 1$,
%$K_n$ is compact and $\mu(K_n) < \infty$, and where
%(the centre of $\Mul$) $Z(\Mul) \cong L_{\infty}(T, \mu)$.
%Let $\{ \h_t \}_{t \in T}$, $\{ \Mul_t \}_{t \in T}$ be the corresponding
%direct integral decompositions of $\h$ and $\Mul$ respectively.
Keeping the above notation, we can moreover assume that $\h$ is infinite dimensional and therefore
that $\h_t$ is a separable infinite dimensional Hilbert space and $\Mul_t$ is a properly infinite
factor for $\mu$-a.e. $t \in T$. Note that by \cite{ChoiEffrosDuke}, if $\Mul$ is semidiscrete
then $\Mul_t$ is semidiscrete for $\mu$-a.e. $t \in T$.

Let $\tilde{\h}$ be a separable infinite dimensional Hilbert space. By \cite{KadisonVolume2} Lemma
 14.1.23, let $\{ x_t \}_{t \in T}$ be a family of maps such that the following hold:
\begin{enumerate}
\item $x_t : \h_t \rightarrow \tilde{\h}$ is a unitary isomorphism for
all $t \in T$.
%\item Let $T \mapsto B(\tilde{\h}) : t \mapsto c(t)$ be an (essentially)
%norm-bounded map.
%Then $t \mapsto c(t)$ is $\mu$ weakly measurable if and only if
%$\{ x_t^* c(t) x_t \}_{t \in T}$ is a decomposable operator
%on $\h$ (with respect to the decomposition
%$\{ \h_t \}_{t \in T}$).
%In the above, ``weakly" means that $B(\tilde{\h})$ is given the weak operator topology.
\item For each measurable field $\xi : t \in T \mapsto \xi(t) \in \h_t$, the map
$t \in T \mapsto x_t \xi(t) \in \tilde{\h}$ is measurable (i.e., for all $\nu \in \tilde{\h}$,
 the map $t \mapsto (x_t \xi(t) | \nu)$ is measurable).
\item For each measurable field $y : t \in T \rightarrow y(t) \in \Mul_t$, the map
$t \in T \mapsto x_t y(t) x_t^* \in \mathbb{B}(\tilde{\h})$ is measurable (i.e., for all $\xi,
\nu \in \tilde{\h}$, the map $t \mapsto (x_t y(t) x_t^* \xi | \nu)$ is measurable).
\end{enumerate}

Let $\{ w_k \}_{k=1}^{\infty}$ be a countable set of unitaries which is strongly dense in
$U(\Mul)$. Hence, $\{ w_k(t)) \}_{k=1}^{\infty}$ is strongly dense in $U(\Mul_t)$ for $\mu$-almost
every $t \in T$. For all $k \geq 1$, by changing $w_k$ on a Borel null subset of $T$ if necessary,
we may assume that $t \mapsto x_t w_k (t) x_t^*$ is Borel.

% Let $\overline{S}$ be the closed unit ball of $B(\tilde{\h})$, given
%the weak operator topology.  Note that
%$\overline{S}$ is a Polish space.  Also, the unitary group
%$U(\tilde{\h})$, given the weak operator topology, is also a Polish space.

Let $\{ \rho_t \}_{t \in T}$, $\{ \phi_t \}_{t \in T}$, $\{ \psi_t \}_{t \in T}$ be the
corresponding direct integral decompositions of $\rho$, $\phi$, $\psi$ respectively. Changing $\{
\phi_t \}_{t \in T}$ and $\{ \psi_t \}_{t \in T}$ on a Borel null set if necessary, we may assume
that for $1 \leq i \leq n$ and $k \geq 1$, the maps %$T \rightarrow \overline{S} :
$t \mapsto x_t \phi_t(a_i) x_t^*$, %$T \rightarrow \overline{S} :
$t \mapsto x_t \psi_t(a_i) x_t^*$ and %$\Phi_{i,k} : $T \rightarrow [0, \infty) :
$t \mapsto |\rho_t( w_k(t) \phi_t(a_i) w_k(t)^* - \psi_t(a_i))|$ are Borel.

Recall (see \cite{KechrisBook} Theorem 14.12) that if $X$ and $Y$ are standard Borel spaces then a
map $f : X \rightarrow Y$ is Borel if and only if its graph is Borel.

%Hence, for each $k \geq 1$, $T_k =_{df} \bigcap_{i=1}^n \Phi_{i,k}^{-1}([0, \epsilon/2))$ is a
%Borel subset of $T$. Hence, $T_k \times U(\tilde{\h})$ is a Borel subset of $T \times
%U(\tilde{\h})$.

%   Now for each $k \geq 1$, let $\Psi_k : T \rightarrow T \times
%U(\tilde{\h}) : t \mapsto (t, x_t w_k(t) x_t^*)$. Then $\Psi_k$ is Borel measurable.  Hence, since
%$T$ is a Polish space, $\Psi_k(T)$ is an analytic subset of $T \times U(\tilde{\h})$. Hence, $\s_k
%=_{df} \Psi_k(T) \cap (T_k \times U(\tilde{\h})) $ is analytic.

Hence, for all $k \geq 1$,
\begin{eqnarray*}
& & \s_k \\
& =_{df} & \bigcap_{i=1}^n \{ (t, x_t w_k(t) x_t^*) \in T \times U(\tilde{\h})
 : | \rho_t( w_k(t) \phi_t(a_i) w_k(t)^* -
\psi_t(a_i)) | < \epsilon/2 \} \\
& = & graph \{ t \mapsto x_t w_k (t) x_k^* \} \cap (\bigcap_{i=1}^n \{ t \in T : | \rho_t( w_k(t)
\phi_t(a_i) w_k(t)^* - \psi_t(a_i)) | < \epsilon/2 \} \times U(\tilde{\h}))
\end{eqnarray*}
is Borel.

Therefore, $\s =_{df} \bigcup_{k=1}^{\infty} \s_k$ is Borel.
%Then $\s$ is an analytic subset of $T
%\times U(\tilde{\h})$.

By Lemma \ref{lem:DirectIntegralEquivalence}, since $\textrm{W*-rank}(\phi) = \textrm{W*-rank}(\psi)$,
$\textrm{W*-rank}(\phi_t) = \textrm{W*-rank}(\psi_t)$ for $\mu$-a.e. $t \in T$. Hence, by Corollary
\ref{cor:properlyinfinitefactoruniqueness}, For $\mu$-a.e. $t \in T$, there exists a unitary $w'
\in U(\tilde{\h})$ with $x_t^* w' x_t \in \Mul_t$ such that $(t, w') \in \s$.

    Therefore, by \cite{KadisonVolume2} Theorem
14.3.6 (measurable selection principle), there exists a Borel null set $N \subset T$ and a $\mu$
measurable map $T - N \rightarrow U(\tilde{\h}) : t \mapsto u(t)$ such that $(t, u(t)) \in \s$ for
all $t \in T-N$. The map $t \mapsto x_t^* u(t) x_t$ is the decomposition of a unitary $u \in \Mul$
such that
$$\| u \phi(a_i) u^* - \psi(a_i) \|_{w*, \rho} < \epsilon$$
for $1 \leq i \leq n$.

   The proof of the necessity of the condition (i.e., the ``if" direction) is similar to
that of Corollary \ref{cor:properlyinfinitefactoruniqueness}.
\end{proof}

%CHECK THE FOLLOWING ARGUMENT!!  IS THE REDUCTION CORRECT???\\

\begin{prop}  Let $\A$ be a C*-algebra and $\Mul$ a von Neumann algebra.
Suppose that either $\A$ is nuclear or $\Mul$ is semidiscrete.

Let $\phi, \psi : \A \rightarrow \Mul$ be *-homomorphisms.

Then $\textrm{W*-rank}(\phi) = \textrm{W*-rank}(\psi)$ if and only if $\phi$ and $\psi$ are weak* approximately
unitarily equivalent. \label{prop:generalVNAuniqueness}
\end{prop}

\begin{proof}

Let us only prove the statement in the case that $A$ is nuclear (the proof for the case that $\mathcal M$ is semidiscrete is similar).

By Proposition 5 of \cite{ChoiEffros}, for any countable subset $\mathcal S\subseteq \mathcal A$, there is a separable nuclear sub-C*-algebra $\mathcal C$ with $\mathcal S\subseteq C$. Therefore, we may assume that $\mathcal A$ is 
separable. 
Hence, let $\F$ be a countable dense set in $\A_+$.  For each $a \in \F$
and $n \geq 1$, let $v_{a,n}$ be a partial isometry witnessing the
Murray-von Neumann equivalence of the support projections of 
$\phi((a - 1/n)_+)$ and $\psi((a - 1/n)_+)$.
By considering the von Neumann algebra generated by $\phi(A)\cup\psi(A) \cup
\{ v_{a,n} : a \in \F, n \geq 1 \}$, we may assume further that $\mathcal M$ is countably generated. Then $\mathcal M$ is a direct product of von Neumann algebras with separable predual. Since any von Neumann algebra is a direct sum of a finite von Neumann algebra and a properly infinite von Neumann algebra, the proof is reduced to the case that $\mathcal M$ is finite or properly infinite, and it follows from Proposition \ref{prop:II_1uniqueness} and Proposition \ref{prop:separableproperlyinfinitecodomain} respectively.
%
%
%The result follows from Proposition \ref{prop:separableproperlyinfinitecodomain},
%Proposition \ref{prop:II_1uniqueness} and the following statements:
%
%\begin{enumerate}
%\item If $\A$ is a nuclear C*-algebra and $\s \subseteq \A$ a countable
%subset, then there exists a separable nuclear C*-algebra $\C$ such that $\s \subseteq \C \subseteq
%\A$ (\cite{ChoiEffros} Proposition 5).
%\item If $\Mul$ is a semidiscrete von Neumann algebra and
%$\s \subset \Mul$ is a countable subset, then there exists a countably generated semidiscrete von
%Neumann algebra $\N$ such that $\s \subset \N \subseteq \Mul$ (\cite{ElliottNonseparableAFD},
%\cite{ElliottWoods}).
%\item Every countably generated von Neumann algebra is the direct product
%of von Neumann algebras with separable predual.
%\item Every von Neumann algebra is a direct sum of a finite von Neumann
%algebra and a properly infinite von Neumann algebra.
%\end{enumerate}
%
\end{proof}

\section{Uniqueness for C*-algebra codomains}

   We need some ideas that have been useful in recent studies of the Cuntz semigroup.
The following exposition follows \cite{OrRorThi}.   Let $\A$ be a C*-algebra. A \emph{tracial
weight} on $\A$ is an additive function $\tau : \A_+ \rightarrow [0, \infty]$ satisfying
$\tau(\lambda a) = \lambda \tau(a)$ and $\tau(x^* x) = \tau(x x^*)$ for all $a \in \A_+$, $x \in
\A$ and $\lambda \in [0, \infty)$.  A tracial weight $\tau$ is \emph{lower semicontinuous} if
$\tau(a) = lim \tau( a_{\alpha})$ whenever $\{ a_{\alpha} \}$ is a norm-convergent increasing net
with limit $a$. We let $T(\A)$ denote the collection of lower semicontinuous tracial weights on
$\A$.

  Each $\tau \in T(\A)$ induces lower semicontinuous dimension function
  $d_{\tau} : \A_+ \rightarrow [0, \infty]$ given by
  $d_{\tau}(a) =_{df} sup_{\epsilon > 0} \tau(f_{\epsilon}(a))$ for all $a \in \A_+$.  Here,
  for every $\epsilon > 0$, $f_{\epsilon} : [0, \infty) \rightarrow [0, \infty)$ is the unique
  continuous function which is $0$ on $0$, $1$ on $[\epsilon, \infty)$, and
  linear on $[0, \epsilon]$.

\begin{lem}  Let $\A$ be a C*-algebra and let $a, b \in \A$ be positive elements, and
let $supp(a), supp(b) \in \A^{**}$ be their (respective support) projections. Then we have the
following:
\begin{enumerate}
\item If $supp(a)$ is Murray--von Neumann subequivalent to $supp(b)$ in $\A^{**}$ then
$d_{\tau}(a) \leq d_{\tau}(b)$ for all $\tau \in T(\A)$;  and if $supp(a)$ is Murray--von Neumann
equivalent to $supp(b)$ in $\A^{**}$ then $d_{\tau}(a) = d_{\tau}(b)$ for all $\tau \in T(\A)$.
\item Suppose, in addition, that $\A$ is separable.
Then the converse of the above statements hold. I.e.,  if $d_{\tau}(a) \leq d_{\tau}(b)$ for all
$\tau \in T(\A)$ then $supp(a)$ is Murray--von Neumann subequivalent to $supp(b)$ in $\A^{**}$;
and if $d_{\tau}(a) = d_{\tau}(b)$ for all $\tau \in T(\A)$ then $supp(a)$ is Murray--von Neumann
equivalent to $supp(b)$ in $\A^{**}$
\end{enumerate}
\label{lem:OrtegaRordamThiel}
\end{lem}

\begin{proof}
The above are \cite{OrRorThi} Corollary 5.4 and Theorem 5.8.
\end{proof}

Motivated by Lemma \ref{lem:OrtegaRordamThiel} and Definition \ref{df:rank}, we define a notion of
rank for maps between C*-algebras.

\begin{df}
Let $\A$, $\B$ be C*-algebras and let $\phi, \psi : \A \rightarrow \B$ be *-homomorphisms.

We say that $\phi$ and $\psi$ have the same \emph{C*-rank} (and write ``$\textrm{C*-rank}(\phi) =
\textrm{C*-rank}(\psi)$") if for every positive element $a \in \A$ and for every $\tau \in T(\B)$,
$d_{\tau}(\phi(a)) = d_{\tau}(\psi(a))$. \label{df:W*-rank}
\end{df}

With notation as in Definition \ref{df:W*-rank}, let $i : \B \rightarrow \B^{**}$ be the natural
inclusion map.  Note that by Lemma \ref{lem:OrtegaRordamThiel}, if $\B$ is separable then
$\textrm{C*-rank}(\phi) = \textrm{C*-rank}(\psi)$ if and only if $\textrm{W*-rank}(i \circ \phi) = \textrm{W*-rank}(i \circ \psi)$.
Moreover, if we drop the hypothesis that $\B$ is separable, we still have the ``if" direction.

\begin{df}  Let $\A$ be a C*-algebra.
\begin{enumerate}
\item[(a)] $\A$ has the \emph{weak* uniqueness property} if for every
von Neumann algebra $\Mul$, for all *-homomorphisms $\phi, \psi : \A \rightarrow \Mul$,
$\textrm{W*-rank}(\phi) = \textrm{W*-rank}(\psi)$ if and only if $\phi$, $\psi$ are weak* approximately unitarily
equivalent.
\item[(b)]  Suppose, in addition that $\A$ is separable.
$\A$ has the \emph{weak uniqueness property} if for every unital separable
C*-algebra $\B$, for
all *-homomorphisms $\phi, \psi : \A \rightarrow \B$, $\textrm{C*-rank}(\phi) = \textrm{C*-rank}(\psi)$ if and only
if $\phi$, $\psi$ are weakly approximately unitarily equivalent.
\end{enumerate}
\label{df:UniquenessProperty}
\end{df}

\begin{lem}  Let $\A$ be a separable C*-algebra.

  Consider the following statements:
\begin{enumerate}
\item $\A$ is nuclear.
\item $\A$ has the weak* uniqueness property.
\item $\A$ has the weak uniqueness property.
\end{enumerate}

Then (1) $\Rightarrow$ (2) $\Rightarrow$ (3).

\label{lem:C*codomainUniqueness}
\end{lem}

\begin{proof}

That (1) implies (2)
follows from Proposition \ref{prop:generalVNAuniqueness}.

We now prove that (2) implies (3).  We first prove the ``only if" direction of (3).

It suffices to prove the following:

Let $\epsilon > 0$ and a positive linear functional
$\rho \in \B^*$ be given.  Let  $a_1, a_2, ..., a_n \in \A$
be elements with $\| a_i \| \leq 1$ for $1 \leq i \leq n$.
Then there exists a unitary $u \in \B$ such that
$$\| u \phi(a_i) u^* - \psi(a_i) \|_{w*, \rho} < \epsilon$$
for $1 \leq i \leq n$.

Let $i : \B \rightarrow \B^{**}$ be the natural inclusion map. Now $\textrm{C*-rank}(\phi) = \textrm{C*-rank}(\psi)$
and $\B$ is separable.  Hence, by Lemma \ref{lem:OrtegaRordamThiel}, $rank(i \circ \phi) = rank(i
\circ \psi)$. Also, $\rho$ extends to a unique normal linear functional on $\B^{**}$ (which we
also denote by ``$\rho$"). Hence, by (2), there exists a unitary $v \in \B^{**}$ such that
$$\| v \phi(a_i) v^* - \psi(a_i) \|_{w*, \rho} < \epsilon/10$$
for $1 \leq i \leq n$. Let $\{ u_{\alpha} \}$ be a net of unitaries in $\B$ such that $u_{\alpha}
\rightarrow v$ in the strong operator topology. Hence, for $1 \leq i \leq n$, $u_{\alpha}
\phi(a_i) u_{\alpha}^* \rightarrow v \phi(a_i) v^*$ in the strong operator topology.
So, choosing
$\alpha_0$ sufficiently large, we have that
$$\| u_{\alpha_0} \phi(a_i) u_{\alpha_0}^* - \psi(a_i) \|_{w*, \rho}
< \epsilon$$
for $1 \leq i \leq n$.

   The ``only if" direction for (3) follows from once more considering
the maps $i \circ \phi, i \circ \psi : \A \rightarrow \B^{**}$ and applying Proposition
\ref{prop:generalVNAuniqueness} and Lemma \ref{lem:OrtegaRordamThiel}.
\end{proof}

\begin{rem}
In section 7, we will show that all the statements in Lemma \ref{lem:C*codomainUniqueness} are
equivalent.

We also note that in Lemma \ref{lem:C*codomainUniqueness}, the proof of (1) $\Rightarrow$ (2) does
not require the separability of $\A$.  Also, in the proof of (2) $\Rightarrow$ (3), only
separability of the codomain algebra $\B$ is required.  In other words, if we modify the
definition of the weak uniqueness property (i.e., Definition \ref{df:UniquenessProperty} (b)) to
allow for nonseparable domain C*-algebras but only separable codomain C*-algebras (i.e., allow for
nonseparable $\A$ but only unital separable $\B$) then (2) $\Rightarrow$ (3) (and hence all of
Lemma \ref{lem:C*codomainUniqueness}) will hold without assuming that $\A$ is separable.
\end{rem}

\section{Injectivity and uniqueness}

Recall that for a von Neumann algebra $\Mul$ and a positive normal linear functional $\chi \in
\Mul_*$, $\|. \|_{\chi}$ is the seminorm on $\Mul$ that is given by $\| x \|_{\chi} =_{df}
\chi(x^*x)^{1/2}$ for all $x \in \Mul$.  Recall also that $\| . \|^{\sharp}_{\chi}$ is the
seminorm on $\Mul$ that is given by $\| x \|^{\sharp}_{\chi} =_{df} \frac{\| x \|_{\chi} + \| x^*
\|_{\chi}}{2} = \frac{\chi( x^* x)^{1/2} + \chi( x x^*)^{1/2}}{2}$ for all $x \in \Mul$.

\begin{lem}  Let $\Mul$ be a properly infinite von Neumann algebra and let
$\chi \in \Mul_*$ be a normal state. For every $\epsilon > 0$ and
for each finite set $x_1, x_2,
..., x_m$ of unitaries in $\Mul$, there exist two orthogonal projections $P, Q \in \Mul$ and
finitely many elements $z_1, z_2, ..., z_m$ in the closed unit ball of $P \Mul P$ such that the
following hold:

\begin{enumerate}
\item $P \sim Q \sim 1_{\Mul}$
\item $P + Q = 1_{\Mul}$
\item $\| z_k - x_k \|^{\sharp}_{\chi} < 6 \epsilon$ for $1 \leq k \leq m$
\item $\| Q \|^{\sharp}_{\chi} < \epsilon$
\end{enumerate}
\label{lem:LemmaSixPointZero}
\end{lem}

\begin{proof}
We may assume that $\epsilon < 1$.

As $1_{\Mul}$ is properly infinite, there exists (see, for example, \cite{DixmierVNA} Corollary
III.8.6.2) a sequence $\{ e_n \}$ of pairwise orthogonal projections in $\Mul$ with $e_n \sim 1$
for all $n$ and $1 = \sum e_n$, where the sum converges in the strong operator topology. Let $S$
be the family of normal states $S =_{df} \{ \chi, \makebox{ } \chi \circ Ad(x_k), \makebox{ } \chi
\circ Ad(x_k^*) : 1 \leq k \leq m \}$. Let $N \geq 1$ be an integer such that $\sum_{n=1}^N
\rho(e_n) > 1 - \epsilon^2$ for all $\rho \in S$.

Set
\begin{equation} \label{equ:definitionofPandQ} 
P =_{df} \sum_{n=1}^N e_n\quad\textrm{and}\quad Q =_{df} 1 - P. 
\end{equation}
Then $\rho(P) > 1 - \epsilon^2$ and $\rho(Q)
< \epsilon^2$ for $\rho \in S$. In particular, $\| Q \|^{\sharp}_{\chi} = \| Q \|_{\chi} <
\epsilon$.

For all $\rho \in S$ and all $b \in \Mul$ with $0 \leq b \leq 1$, we have by the Cauchy--Schwarz
inequality that $| \rho(P b Q) |$, $| \rho(Q b P) |$, $| \rho(Q b Q) |$ all are strictly less than
$\epsilon^2$, and therefore
\begin{equation} | \rho(b) - \rho(P b P) | < 3 \epsilon^2
\label{equ:PCutdownEstimate}
\end{equation}

Set $z_k =_{df} P x_k P$ for $1 \leq k \leq m$.  Then
$$\| x_k - z_k \|_{\chi}^{\sharp} \leq \| P x_k Q \|_{\chi}^{\sharp} +
\| Q x_k P \|_{\chi}^{\sharp} + \| Q x_k Q \|_{\chi}^{\sharp}$$

As $\| a Q \|_{\chi} \leq \| a \| \| Q \|_{\chi}$ for $a \in \Mul$, we have for $1 \leq k \leq m$,
that
$$\| x_k - z_k \|_{\chi}^{\sharp} \leq  2 \| Q \|_{\chi} + (1/2)(\| Q x_k P \|_{\chi} +
\| Q x_k^* P \|_{\chi} )$$

From this, (\ref{equ:PCutdownEstimate}) and (\ref{equ:definitionofPandQ}),
for $1 \leq k \leq m$,
\begin{eqnarray*}
& & \| x_k - z_k \|_{\chi}^{\sharp} \\
& < &  2 \| Q \|_{\chi} + (1/2)( \rho(x_k^* Q x_k) + \rho(x_k Q x_k^*) ) + 3 \epsilon^2 \\
& < & 2 \epsilon + (1/2)( \epsilon^2 + \epsilon^2) + 3 \epsilon^2 \\
& < & 6 \epsilon.
\end{eqnarray*}
\end{proof}

\begin{thm} Let $\Mul$ be a
properly infinite von Neumann algebra.
Then the following statements are equivalent:

\begin{enumerate}
\item $\Mul$ is semidiscrete.
\item Let $\A$ be a C*-algebra and $\phi, \psi : \A \rightarrow \Mul$
two *-homomorphisms. Then $\textrm{W*-rank}(\phi) = \textrm{W*-rank}(\psi)$ if and only if $\phi$ and $\psi$ are
weak* approximately unitarily equivalent.
\end{enumerate}
\label{thm:InjectivityCharacterization}
\end{thm}

\begin{proof}

   That (1) implies (2) follows from Proposition
\ref{prop:generalVNAuniqueness}.

   We now prove that (2) implies (1).

     We now prove the ``if" direction.
By \cite{ElliottNonseparableAFD} and \cite{ElliottWoods}, it suffices to prove that for a normal
state $\chi \in \Mul_*$, for any finite set of unitaries $x_1, x_2, ..., x_m \in U(\Mul)$ and any
$\epsilon > 0$, there exists a finite dimensional unital subalgebra $F$ of $\Mul$ and elements
$y_1, y_2, ..., y_m \in F$ such that
$$\| x_k - y_k \|^{\sharp}_{\chi} < \epsilon$$
for $1 \leq k \leq m$.

By Lemma \ref{lem:LemmaSixPointZero}, let $P, Q \in \Mul$ be orthogonal projections and $z_k$ (for
$1 \leq k \leq m$) an element of $P \Mul P$ with norm less than or equal to one such that
\begin{equation} \label{equ:Barcelona} \end{equation}
\begin{enumerate}
\item $P \sim Q \sim 1_{\Mul}$
\item $P + Q = 1_{\Mul}$
\item $\| z_k - x_k \|^{\sharp}_{\chi} < \epsilon/2$ for $1 \leq k \leq m$
and
\item $\| Q \|^{\sharp}_{\chi} < \epsilon/2$.
\end{enumerate}

   Let $\A$ be the unital C*-subalgebra of $P \Mul P$ that is generated
by $\{ P, z_1, z_2, ..., z_m \}$. Let $\gamma : \A \rightarrow Q \Mul Q$ be a full unital
*-homomorphism. ($\gamma$ is \emph{full} means that for every positive $a \in \A - \{ 0 \}$,
$\gamma(a)$ is a full element of $Q \Mul Q$.) Let $\phi : \A \rightarrow \Mul$ be the unital full
*-homomorphism that is given by $\phi(a) =_{df} a \oplus \gamma(a)$ for all $a \in \A$.

Since $\Mul$ is properly infinite, $\Mul = \N \overline{\otimes} \mathbb{B}(\h)$ where $\N$ is a
von Neumann algebra and $\h$ is a separable infinite dimensional Hilbert space.  Hence, let $\psi
: \A \rightarrow 1_{\N} \otimes \mathbb{B}(\h)$ be any unital full *-homomorphism.

Since $\phi$ and $\psi$ are both full, $\textrm{W*-rank}(\phi) = \textrm{W*-rank}(\psi)$. Hence, by hypothesis,
$\phi$ and $\psi$ are weak* approximately unitarily equivalent. Hence, by Lemma
\ref{lem:EquivalentEquivalences}, $\phi$ and $\psi$ are $\sigma$-strong* approximately unitarily
equivalent. Hence, there exists a net $\{ u_{\alpha} \}_{\alpha \in I}$ of unitaries in $\Mul$
such that for all $a \in \A$, $u_{\alpha} \psi(a) u_{\alpha}^* \rightarrow \phi(a)$ in the
$\sigma$-strong* topology. In particular, for $1 \leq k \leq m$, $u_{\alpha} \psi(z_k)
u_{\alpha}^* \rightarrow \phi(z_k) = z_k \oplus \gamma(z_k)$ in the $\sigma$-strong* topology.
From this and (\ref{equ:Barcelona}), there exists $\alpha_0 \in I$ such that for $1 \leq k \leq
m$,
$$\| x_k - u_{\alpha_0} \psi(z_k) u_{\alpha_0}^* \|^{\sharp}_{\chi}
< \epsilon.$$

   Now since $1_{\N} \otimes \mathbb{B}(\h)$ is approximately
finite dimensional,  there exists a sequence $\{ F_n \}$ of finite dimensional unital
C*-subalgebras of $1_{\N} \otimes \mathbb{B}(\h)$ and for $1 \leq k \leq m$, there exists an
element $w_{k, n} \in F_{n}$ such that $w_{k, n} \rightarrow \psi(z_k)$ in the $\sigma$-strong*
topology. Hence, there exists $N \geq 1$ such that for $1 \leq k \leq m$,
$$\| x_k - u_{\alpha_0} w_{k, N} u_{\alpha_0}^* \|^{\sharp}_{\chi}
< \epsilon$$ and $u_{\alpha_0} w_{k, N} u_{\alpha_0}^*$ is an element of the finite dimensional
C*-algebra $u_{\alpha_0} F_{N} u_{\alpha_0}^*$.

\end{proof}

For finite von Neumann algebras, we have the following:

\begin{thm}
Let $\Mul$ be a type $II_1$ factor with separable predual, and (unique)
faithful normal tracial state $\tau$.
Then the following statements are equivalent:
\begin{enumerate}
\item $\Mul$ is injective.
\item For every $C^*$-algebra $\A$,
for all  $*$-homomorphisms $\phi, \psi : \A \rightarrow \Mul \overline{\otimes} \Mul$,
$\textrm{W*-rank}(\phi) = \textrm{W*-rank}(\psi)$ if and only if $\phi$ and $\psi$ are $\sigma$-strong*
approximately unitarily equivalent.
\end{enumerate}
\end{thm}

\begin{proof}
That (1) implies (2) follows from Proposition \ref{prop:II_1uniqueness} and
the fact that if $\Mul$ is the injective type $II_1$ factor then
$\Mul \overline{\otimes} \Mul$ is also the injective type $II_1$ factor.

We now prove that (2) implies (1). By \cite{ConnesInjective}, it is enough to show that the flip
automorphism $\sigma \in Aut(\Mul \overline{\otimes} \Mul)$, which is given by $\sigma(x \otimes
y) = y \otimes x$ ($x, y \in \Mul$), is approximately inner.

  Let $\A$ be a separable unital $C^*$-subalgebra of $\Mul \overline{\otimes}
\Mul$ such that $\A$ is $\| . \|_{\tau}$-dense in $\Mul \overline{\otimes} \Mul$.

  Note that on $\A$, $\tau \circ \sigma = \tau$.
Hence, by assumption, there exists a sequence $\{ u_n \}_{n=1}^{\infty}$
of unitaries in $\Mul \overline{\otimes} \Mul$ such that for all $a \in \A$,
$$\| \sigma(a) - u_n a u_n^* \|_{\tau} \rightarrow 0$$
as $n \rightarrow \infty$.

    As $\Mul$ is type $II_1$ and $\A$ is $\| . \|_{\tau}$-dense in
$\Mul \overline{\otimes} \Mul$, it follows that $\sigma$ is approximately inner, as required.
\end{proof}

\section{Nuclearity and uniqueness}

\begin{thm}  Let $\A$ be a separable C*-algebra.
Then the following are equivalent:
\begin{enumerate}
\item $\A$ is nuclear.
\item $\A$ has the weak* uniqueness property.
\item $\A$ has the weak uniqueness property.
\end{enumerate}
\label{thm:nuclearitycharacterization}
\end{thm}

\begin{proof}

  The directions $(1) \Rightarrow (2) \Rightarrow (3)$ follow from
Lemma \ref{lem:C*codomainUniqueness}.

   We now prove that (3) implies (1).

   Let $\h$ be a separable infinite dimensional Hilbert space. We will first prove that the von Neumann algebra
$\Mul =_{df} \A^{**} \otimes \mathbb{B}(\h)$ is injective (where the tensor product is the spatial
tensor product).

   By \cite{ElliottWoods} (see also \cite{ElliottNonseparableAFD}), it suffices to prove the following:  Let $\epsilon > 0$ and let
$\rho \in \Mul_*$ be a normal state. Let $x_1, x_2,...., x_m \in \Mul$ be elements such that for
$1 \leq k \leq m$, $x_k = a_k \otimes b_k$ where  $a_k \in \A$, $b_k$ compact operators in
$\mathbb{B}(\h)$ and $\| a_k \|, \| b_k \| \leq 1$. Then there exists a finite dimensional unital
C*-subalgebra $F \subset \Mul$ and $y_k \in F$ (for $1 \leq k \leq m$) such that
$$\| x_k - y_k \|^{\sharp}_{\rho} < \epsilon$$
for $1 \leq k \leq m$.

    Since $\mathbb{B}(\h)$ is injective, we may assume that
there exists a finite dimensional simple
C*-subalgebra $F_0 \subseteq \mathbb{B}(\h)$ such that
$b_k \in F_0$ for $1 \leq k \leq m$ and
$1_{\mathbb{B}(\h)} - 1_{F_0}$ is Murray-von Neumann equivalent to $1$ in
$\mathbb{B}(\h)$.
Let $P =_{df} 1_{\Mul} - 1_{\A^{**}} \otimes 1_{F_0}$.
We may also assume that $\rho(1_{\Mul} -  P)
< \epsilon/3$.

Suppose that $N \geq 1$ is such that $F_0 \cong \M_N$ (N by N matrices over complex numbers). Let
$\tilde{\A} =_{df} \A + \mathbb{C} 1_{\A^{**}}$.  (Note that if $\A$ is unital then $\tilde{\A} =
\A$.) Let $\{ e_{i,j} \}_{1 \leq i,j \leq N}$ be a system of matrix units for $\M_N$ (i.e., for
$F_0$). Consider the natural map $i : \tilde{\A} \otimes \M_N \rightarrow \tilde{\A} \otimes F_0 :
a \otimes e_{i,j} \mapsto a \otimes e_{i,j}$ for all $a \in \tilde{\A}$ and for $1 \leq i,j \leq
N$. Then $x_k \in i(\A \otimes \M_N)$ for $1 \leq k \leq m$. Let $\gamma : \tilde{\A} \otimes \M_N
\rightarrow P \Mul P$ be a full unital *-homomorphism. (Recall that $\gamma$ is \emph{full} means
that for all positive $a \in \M_N(\tilde{\A}) - \{ 0 \}$, $\gamma(a)$ is a full element of $P \Mul
P$.)

   Let $\phi : \M_N(\tilde{\A}) \rightarrow \Mul$ be the full unital
*-homomorphism that is given by $\phi=_{df} i \oplus \gamma$. Note that
\begin{equation}
\| \phi(x_k) - x_k \|^{\sharp}_{\rho} < \epsilon/3
\label{equ:mapapproximation}
\end{equation}
for $1 \leq k \leq m$.

Let $\psi : \M_N(\tilde{\A}) \rightarrow 1_{\A^{**}} \otimes \mathbb{B}(\h)$ be a full unital
*-homomorphism. For $1 \leq i,j \leq N$, let $p_{i,j} =_{df} \phi(1_{\A^{**}} \otimes e_{i,j})$
and $q_{i,j} =_{df} \psi(1_{\A^{**}} \otimes e_{i,j})$.  Since $\phi$ and $\psi$ are both full
maps, $p_{i,i}$ is Murray--von Neumann equivalent to $q_{i,i}$ in $\Mul$ for $1 \leq i \leq N$.
Thus, conjugating $\psi$ (and hence conjugating $1_{\A^{**}} \otimes \mathbb{B}(\h)$) by a unitary
if necessary, we may assume that $p_{i,j} = q_{i,j}$ for $1 \leq i,j \leq N$.

Let $\C \subseteq \Mul$ be a separable C*-subalgebra such that the following hold:

\begin{enumerate}
\item $\C$ is unital and $1_{\C} = p_{1,1}$,
\item $\C$ is properly infinite,
\item  $\phi(\tilde{\A} \otimes e_{1,1})$ and
$\psi(\tilde{\A} \otimes e_{1,1})$ are both full C*-subalgebras of $\C$.
\end{enumerate}

Since $\phi |_{\tilde{\A} \otimes e_{1,1}}$ and $\psi |_{\tilde{\A} \otimes e_{1,1}}$ are both
full maps into $\C$ and since $\C$ is properly infinite, we have that $\textrm{C*-rank}(\phi |_{\tilde{\A}
\otimes e_{1,1}}) = \textrm{C*-rank}(\psi |_{\tilde{\A} \otimes e_{1,1}})$ (as maps into $\C$). Hence,
since $\A$ has the weak uniqueness property, $\phi |_{\tilde{\A} \otimes e_{1,1}}$ and $\psi
|_{\tilde{\A} \otimes e_{1,1}}$ are weakly approximately unitarily equivalent (as maps into $\C$).
Hence, as maps into $\Mul$, $\phi$ and $\psi$ are weak* approximately unitarily equivalent.

Since $\phi$ and $\psi$ are unital maps, by Lemma \ref{lem:EquivalentEquivalences}, $\phi$ and
$\psi$ are $\sigma$-strong* approximately unitarily equivalent (as maps into $\Mul$). From this
and (\ref{equ:mapapproximation}), let $u \in \Mul$ be a unitary such that
\begin{equation}
\| x_k - u \psi(x_k) u^* \|^{\sharp}_{\rho} < \epsilon/3
\label{equ:Boudreaux}
\end{equation}
for $1 \leq k \leq m$.

   Since $\mathbb{B}(\h)$ is an injective von Neumann algebra,
there exists a sequence $\{ F_{n} \}$ of finite dimensional C*-subalgebras of $1_{\A^{**}} \otimes
\mathbb{B}(\h)$ and elements $c_{n, k} \in F_n$ (for $1 \leq k \leq m$) such that $c_{n, k}
\rightarrow \psi(x_k)$ in the $\sigma$-strong* topology for $1 \leq k \leq m$. Note that since
$Ad(u) \rho \in \Mul_*$, we must have that $\| u \psi(x_k) u^* - u c_{n, k} u^* \|^{\sharp}_{\rho}
= \| \psi(x_k) - c_{n, k } \|^{\sharp}_{Ad(u) \rho} \rightarrow 0$. From this and
(\ref{equ:Boudreaux}), there exists $n_0 \geq 1$ such that
$$\| x_k - u c_{n_0, k} u^*  \|^{\sharp}_{\rho} < \epsilon$$
for $1 \leq k \leq m$. Since $u F_{n_0} u^*$ is finite dimensional, we are done.

   Since $\A^{**} \otimes \mathbb{B}(\h)$ is injective,
$\A^{**}$ is injective.  Hence, by \cite{ChoiEffros},  $\A$ is nuclear.
\end{proof}

\begin{rem}
In Theorem \ref{thm:nuclearitycharacterization}, the direction  (2) $\Rightarrow$ (1) does not
require separability of $\A$.  The proof is a modification of the proof of (3) $\Rightarrow$ (1).

Indeed, if we modify Definition \ref{df:UniquenessProperty} (b) (the definition of the weak
uniqueness property) to allow for nonseparable $\A$ (though still separable unital codomains
$\B$), the proof of (3) $\Rightarrow$ (1) would still work with minor modifications.  From this
and the remarks after Lemma \ref{lem:C*codomainUniqueness}, we have that, with a modification of
Definition \ref{df:UniquenessProperty}, all the statements in Theorem
\ref{thm:nuclearitycharacterization} are equivalent.
%We have not explicitly made this modification
%since it seems to us to not be the natural notion.
\end{rem}

\end{document}